\documentclass[11pt,a4paper]{article}
\usepackage{geometry} 
\usepackage{algorithmicx}
\usepackage{amsmath,amssymb}
\usepackage{mathrsfs}
\usepackage{tabularx}
\usepackage{textcomp}
\usepackage{gensymb}
\usepackage{comment}
\usepackage{bbold} 
\usepackage[metapost]{mfpic}
\usepackage{pgfplots}
\usepackage{pgfplotstable}
\pgfplotsset{compat=newest}
\usepackage{graphicx}
\usepackage{tikz}
\usetikzlibrary{arrows.meta, bending, positioning,calc}
\usepackage{eso-pic}
\usetikzlibrary{tikzmark,fit,calc}
\usepackage{caption}
\usepackage{pstricks,pst-plot}
\usepackage{pstricks-add,pst-func} 
\usepackage{listings}
\usepackage[toc,page]{appendix}
\usepackage{xcolor,pict2e,curve2e}
\usepackage{nicematrix}
\usepackage{pst-node}
\usepackage{xcolor}
\usepackage{setspace}
\usepackage{colortbl}
\usepackage{femtikz}
\usepackage[symbol]{footmisc}

\newtheorem{thm}{Theorem}
\newtheorem{prop}{Proposition}
\newtheorem{lem}{Lemma}
\newtheorem{rem}{Remark}

\newcommand\ddfrac[2]{\frac{\displaystyle #1}{\displaystyle #2}}

\newcommand{\ba}{\mathbf a}
\newcommand{\bb}{\mathbf b}

\newcommand{\bfn}{\mathbf n}

\newcommand{\bp}{\mathbf p}

\newcommand{\bx}{\mathbf x}

\newcommand{\by}{\mathbf y}

\newcommand{\bfnu}{\boldsymbol\nu}

\def\cN{\mathcal N}

\def\cP{\mathcal P}
\def\cT{\mathcal T}
\def\cE{\mathscr E}

\def\omnu{\omega_{\bfnu}}
\def\cI{\mathcal I}
\def\cJ{\mathcal J}

\title{A quasi-interpolation method for finite element spaces}
\author{Ohannes  Karakashian\thanks{Department of Mathematics, The University of Tennessee, Knoxville, TN 37996 (okarakas@utk.edu).}}
\date{\today}

\begin{document}

\maketitle

\abstract{
A combination of Taylor polynomial approximation and averaging results in a 
powerful quasi-interpolation technique that can be used to analyze the approximation 
properties of finite element spaces.
With this particular approach, the notion of affine equivalent families is completely bypassed
in the sense that the degrees of freedom are defined locally rather than being exported
from a reference element. In particular, domains with curved boundaries can be handled
in a natural manner. Finally, the ease of its application to a particular finite element
makes it a valuable pedagogical tool.
}

\section{Introduction}\label{sec1}
The approximation of functions by piecewise polynomials is a core problem in 
general approximation theory and is an issue of central importance in
finite elements. The influential textbooks of P. Ciarlet \cite{Ciarlet}
and Brenner and Scott \cite{BS2008} offer a treatment which can be considered as comprehensive
in their exposition of some of the main mathematical tools that are available.
In essence, the approach hinges on deriving approximation results of polynomial preserving operators
and it has become a common practice to refer to this approach as the Bramble-Hilbert Lemma.
The central result of all this is a theorem of Deny-Lions \cite{DL1953} which states that over the quotient
space $H^{k+1}(\Omega) /P_k(\Omega)$ the semi-norm $|\cdot|_{k+1,\Omega}$ is equivalent to
the quotient norm. (See Theorem 3.1.1. of \cite{Ciarlet} for a generalization to the spaces $W^{m,p}(\Omega)$).
One important aspect of a further development of this approach is to use a scaling argument between
a ``master'' element and the local elements to obtain bounds of the errors in terms of the discretization
parameter. This in turn requires that a particular finite element be imbedded in an affine family. 

There are several shortcomings associated with this method. First, in the case of affine Lagrange
elements, to mention one such case, it is not possible to obtain a bound of the form 
$\|u-\chi\|_{L^2(\Omega)} \le c h|u|_{1,\Omega}$, given that for afffine polynomials the quotient
norm is $|\cdot|_2 $. Indeed, bounds for affine elements depending only the $|\cdot|_1$  seminrom
play an essential role in a posteriori error estimation. Indeed, special interpolation procedures have been
developed to address this particular issue. We mention the Cl\'{e}ment interpolant \cite{Clement75}
and the Scott-Zhang interpolant \cite{SZ90}. 

Another issue that arises is the fact that certain finite elements, e.g. those involving normal derivatives
as degrees of freedom, do not form an  affine equivalent family, hence special treatments must be applied
to derive bounds for the approximation errors. Ciarlet introduces the concept of {\it Almost-affine} 
families of finite elements and applies it to the Argyris element. The required technicalities are by
no means trivial, requiring 3 full pages (cf. \cite{Ciarlet}, pp.  337-339). Also, it dos not appear
that finding an associated almost-affine family is always straightforward.

In this paper we outline an approach to approximation that offers wide applicability to many types of finite elements
that are used in practice. It is also characterized by the simplicity of its application to a given particular element. This is 
due mainly to three basic tools.that are well known having been used in a variety of contexts: Cell approximation by
Taylor polynomials, a trace estimate and Oswald averaging. Therefore the novelty resides in the combination rather
than the development of the basic tools.

\begin{enumerate}
\item The underlying procedure for approximating a function $u$ is that of the Taylor
polynomial defined locally on each cell of the mesh.
This is feasible in view of the density of smooth functions in Sobolev spaces. We thus construct
a space $V_h$ of discontinuous piecewise polynomial functions having optimal approximation properties in a variety of norms.
An important contribution in this vein is the paper of Dupont and Scott where they construct an averaged
Taylor polynomial \cite{DS80}. The paper of Duran \cite{Duran} has estimates for an averaged Taylor polynomial
but using the Hardy-Littlewood maximal function in the proof. 
We also mention the paper of Ciarlet-Wagschal \cite{CW71} on multipoint Taylor approximations.
The proof that we present provides an estimate of the remainder of the
Taylor approximation directly and does not involve the averaging process used in the two reference cited above.
It is an interesting fact that while the Taylor polynomial can be defined at every point of the cell,
the remainder, on the other hand, may not be bound optimally at every point due to a phenomenon of concentration.
In the paper \cite{BJK90} a vector Taylor polynomial  was constructed as an
approximation for locally solenoidal vector fields. Let us say that one advantage of using Taylor's
Theorem is that it applies to general domains, read individual cells of the mesh, including those 
with curved boundaries.

\item The global finite element space $V_h$, eventually used in the actual computations, is obtained by {\em stitching}
the spaces of local element-wise functions. This is standard in finite element theory and practice and relies in a 
fundamental way on the ability of the local degrees of freedom to form a compatible global network.
The process that is employed next takes averages of the local degrees of freedom to construct an element of $\chi \in V_h$. .
This is an important idea of  P. Oswald \cite{oswald93} and the process is referred to as Oswald interpolation. 
Similar averaging techniques were used by S. Brenner also in the context of preconditioning \cite{brennerxx}.
The averaging process was also used in \cite{KP03} as a means to a posteriori error estimation.

\item The difference between the Taylor polynomials and the function $\chi$ that is obtained by averaging
can be expressed as jumps of the local Taylor polynomials on the edges or faces 
that emanate from vertices of cells. To the best of our knowledge
the paper \cite{KP03} was the first to give a constructive proof for Lagrangian elements.
The final step is to use well known trace estimates to bound these jumps in terms of the difference
$u-v_h$ in each cell, closing the circle so to speak.
\end{enumerate}

It is proper to discuss our results in the context of the paper of  A. Ern and J.-L. Guermond  \cite{EG2017}. 
They set up a powerful panoply of techniques and they construct a local interpolation operator 
$\cI_K^\sharp$ and prove $L^1$ stability and optimal approximation results.
They also construct a quasi-interpolation operator 
$\cJ_h^{av}$ which is a combination of $\cI_K^\sharp$ and Oswald type averaging 
and prove local and global approximation results.
Our approach is closer to the latter but there are marked differences.
First, we use Taylor's Theorem to generate a discontinuous polynomial  approximations $v_h$ to the
unknown function $u$ and then apply averaging to construct the finite element approximations.
Two aspects we mention: First, Taylor's Theorem allows approximations on domains with curved
boundaries. It is worth mentioning that global convexity of the is not required.

Another significant contrast between other references, including \cite{EG2017}, and the present work is that we 
completely avoid the use of reference or so-called {\em master} elements. 
To be more precise, we do not assume that the local degrees of freedom are affine exports from the reference
finite element. While the focus is not at present on the interpolant, avenues to possibilities hitherto unknown
may be opened, say to meshes consisting of a mix of cells of different geometries.  

\vskip 6pt

The paper is organized as follows: A modicum of notation, assuming that the reader is familiar with the theory
of Sobolev spaces and the masterful presentation of finite elements contained in Ciarlet's book \cite{Ciarlet}.
A proof of Taylor's Theorem as an approximation device is presented.
It follows a different approach from \cite{DS80} by estimating the remainder and makes
essential use of the Hardy-Littlewood maximal function theorem.  Next, in Theorem 2, a trace estimate for
functions in $H^1(D)$ is presented with a correct balance of the $L^2$ norm and the $H^1$ seminorm in the
interior. for the proof we refer to \cite{KJ98}. Both results hold under the assumption that the cells of the mesh are starlike Lipschitz domains. 
Proposition 1 shows a bound on the sizes of the canonical basis functions on cells. It appears to be new.
Proposition 2 is the cornerstone of the entire approach. Using the Oswald averaging process, a
function in the space $\cP_m(\cT_h)$ of discontinuous piecewise polynomial functions can be approximated
by a function $\chi$ in the global finite element space $V_h$. A version specific for Lagrangian elements
had appeared in \cite{KP03} and is essentially the same as Lemma 4.2 of \cite{EG2017}. 
These being all the necessary technical steps of our approach, we apply them to the three examples 
of the Courant triangle, The Argyris triangle and the Hsieh-Clough-Tocher macro-element. 
In particular, the HCT element is an interesting case of a non-polynomial finite element.
The purpose here is to illustrate the efficacy of the method. We could also extol the pedagogical merits
of its simplicity with this paper containing all the necessary tools that are required. 
Furthermore, some of the proof, while available elsewhere, do not ppear to be widely known
e.g. the development of the trace estimate.

\section{Preliminaries}\label{sec2} 

We shall operate in the context of Hilbertian Sobolev
spaces deferring  the development of estimates in the general $W^{m,p}$ spaces to future work.
The notation we adopt is standard in Sobolev space theory with $\|\cdot\|_K$ denoting the Euclidean
norm on $K$ and $|\cdot|_{l,K}$ denoting the seminorm in $H^l(K)$. We shall also
use the widely adopted notation of {\it broken} Sobolev spaces $H(\cT_h)$.
\vskip 6pt

{\sc Definition}
A Lipschitz domain $D \subseteq R^n$ is said to be {\it starlike} if there is a point $\bp$ in its interior such that the line segment
$[\bp,\bx]$ is contained in $D$ for any $\bx \in D$.
\vskip 3pt

We make the following assumptions
\begin{enumerate}
\item[(A1)] the set $S \subseteq D$ of points with respect to which $D$ is starlike has
positive $n$-dimensional measure. 
\item[(A2)] with the unit outward normal $\bfnu$ to $D$ existing almost everywhere on $\partial D$,
there exist a constant $c$ and a point $\bp \in S$ such that
\[
(\bx-\bp)\cdot \bfnu \ge c \,  {\rm diam}(D) \quad \mbox{ for almost all } \bx \in \partial D.
\]
\end{enumerate}

We have in mind sets $D$ that are cells in a typical partition of a domain $\Omega$ such as are used
in the finite element method. The assumptions that have been made are clearly satisfied for such sets.
Moreover, for such sets, assumption (A2) implies that they are {\it shape-regular} as understood in terms of 
The diameters of the inscribed and circumscribed circles being commensurate.
Consequently, ``thin'' sets are excluded from the present discussion.
\vskip 3pt

The following result shows that functions in $H^m(D)$ can be approximated by a Taylor polynomial.
The proof which is adapted from \cite{BJK90} is more direct than the development of the averaged
Taylor polynomial of Dupont and Scott \cite{DS80}.
\vskip 6pt

\begin{thm}\label{Taylor}  Let $u \in H^{k+1}(D), k\ge 0$ with a starlike and shape-regular set $D$.
There exists a polynomial $\chi$ of total degree $k$ such that
\begin{equation}\label{Taylor0}
|u-\chi|_{l,D} \le c(k+1) h^{k+1-l} |u|_{k+1,D}, \quad l=0,\dots,k.
\end{equation}
The constant $c$ depends only on $n,D$ and the ratio $\mu(D)/\mu(S)$.

{\bf Proof}
It suffices to consider $u$  in $C^\infty(\overline D)$ given its
density in $H^{k+1}(D)$ under the prevailing assumptions. 
let $\chi = T_{\bp}^k[u]=\sum_{|\alpha|\le k} \ddfrac{\partial^\alpha u(\bp)}{\alpha !}(\bx-\bp)^\alpha$ 
for $\bp \in S$ be the Taylor polynomial in $\bx$ of total degree $k$. 
We will show below that a point $\bp\in S$ exists for which the error estimates hold.
For the remainder term $e = u-T_{\bp}^k[u]$, we have
\begin{equation}\label{Taylor1}
e(\bx) = (k+1) \sum_{|\alpha|=k+1} \ddfrac{1}{\alpha!} \int_0^1 (1-t)^k
 \partial^\alpha u\big(\bp+t(\bx-\bp)\big)(\bx-\bp)^\alpha dt.
\end{equation}
Using the Schwarz inequality and Fubini's Theorem
\begin{equation}\label{Taylor2}
\begin{split}
\| e\|_D^2 & = (k+1)^2 \int_D d \bx \bigg( \int_0^1 (1-t)^k \sum_{|\alpha|=k+1} \ddfrac{1}{\alpha!}
     \partial^\alpha u\big(\bp+t(\bx-\bp)\big)(\bx-\bp)^\alpha dt \bigg)^2 \\
 & \le (k+1)^2 \int_D d \bx  \int_0^1 \bigg| \sum_{|\alpha|=k+1} \ddfrac{1}{\alpha!}
     \partial^\alpha u\big(\bp+t(\bx-\bp)\big)(\bx-\bp)^\alpha \bigg|^2 dt \\
     & \le c(k+1) h^{2(k+1)} \int_D d \bx \int_0^1 \sum_{|\alpha|=k+1} |\partial^\alpha u\big(\bp+t(\bx-\bp)\big)|^2  dt \\
     & = c(k+1) h^{2(k+1)} \int_0^1 dt \int_D \sum_{|\alpha|=k+1} |\partial^\alpha u\big(\bp+t(\bx-\bp)\big)|^2 d \bx.
     \end{split}
 \end{equation}
 Let $R(\bx) = \sum_{|\alpha|=k+1}|\partial^\alpha u(\bx)|^2$ 
 and $D_t = \{\by: \by = \bp+t(\bx-\bp), \bx \in D, 0 \le t \le 1\}$. A change of variables gives
 \begin{equation}\label{Taylor3}
 \int_0^1 dt \int_D \sum_{|\alpha|=k+1}|\partial^\alpha u\big(\bp+t(\bx-\bp)\big)|^2 d\bx = 
 \int_0^1 t^{-n} dt \int_{D_t} R(\bx) d \bx.
 \end{equation}
 We will show that there exists a choice of $\bp$ which will allow control of the term $t^{-n}$.
 
 We extend $R(\bx)$ by zero outside of $D$ and denote the extension also by $R$, an element of
 $L^1(R^n)$. The Hardy-Littlewood Maximal function of $R(\bx)$ is defined as
 \begin{equation}\label{Taylor4}
 M(R)(\bx) = \sup_{r>0} \ddfrac{1}{\mu(B(\bx,r))} \int_{B(\bx,r)} |R(\by)| d \by
 \end{equation}
  $B(\bx,r)$ being the ball of radius $r$ centered at $\bx$ and $\mu$ is Lebesgue measure.
  The following inequality is well-known (cf. \cite{Stein})
  \begin{equation}\label{Taylor5}
  \mu\{\bx: M(R)(\bx) > \alpha \} < \ddfrac{5^n}{\alpha} \int_{R^n} |R(\by)| d\by
  = \ddfrac{5^n}{\alpha} \int_D R(\by)d \by, \quad \forall \alpha >0.
  \end{equation}
  Choosing $\alpha = 2 \ddfrac{5^n \int_D R(\by) d\by}{\mu(S)}$,
  it follows that
  \begin{equation}\label{Taylor6}
  \mu\{ \bx: M(R)(\bx) > \alpha\} < \ddfrac{\mu(S)}{2}.
  \end{equation}
  Consequently, there exists $\bp \in S$ such that 
  \[
  M(R)(\bp) \le \alpha = 2 \ddfrac{5^n \int_D R(\by)d\by}{\mu(S)}.
  \]
  In particular, using such a $\bp$ for the $T_\bp^k[u]$ yields
  \begin{equation}\label{Taylor7}
  \ddfrac{1}{\mu(B(\bp,r))} \int_{B(\bp,r)} R(\bx)d\bx \le 2 \ddfrac{5^n \int_D R(\bx) d\bx}{\mu(S)}, \quad
  \forall r >0.
  \end{equation}
  For $t \in [0,1]$ let $r_t$ be the radius of the smallest ball that contains $D_t$. Observe that $r_t = tr_1$ and
  that $\mu(B(\bp,r_t)) = t^n \mu(B(\bp,r_1)$. Hence, \eqref{Taylor7} implies that
  \[
  \int_{D_t} R(\bx)d\bx \le \int_{B(\bp,r_t)} R(\bx)d\bx \le 2 \cdot5^n t^n 
     \ddfrac{\mu(B(\bp,r_1))}{\mu(D)} \ddfrac{\mu(D)}{\mu(S)} \int_D R(\bx) d\bx.
 \]
 Since $r_1 \le diam(D)$ and $D$ is shape regular, $\mu(B(\bp,r_1))/\mu(D)$ is bounded by a
  constant that depends only on $n$. It then follows from the last inequality that
  \begin{equation}\label{Taylor8}
  \int_{D_t} R(\bx) d\bx \le c t^n \int_D R(\bx)d\bx = c t^n |u|_{k+1,D}^2,
  \end{equation}
  with the constant $c$ depending only on $n$ and the ratio $\mu(D)/\mu(S)$. Using \eqref{Taylor8},
  it follows from \eqref{Taylor2} and \eqref{Taylor3} that \eqref{Taylor0} holds with $l=0$. This also
  takes care of the case $k=0$.
  
  Now consider a multi index $\beta$ with $1\le |\beta| \le k$. It is a simple exercise to 
  show that  $\partial^\beta T_\bp^k[u] = T_\bp^{k-|\beta|}[\partial^\beta u]$. We have
  \[
  \begin{split}
 \partial^\beta e(\bx) & = \partial^\beta\big(u-T_\bp^k[u]\big) = \partial^\beta u - T_\bp^{k-|\beta|}[\partial^\beta u] \\
  & = (k+1-|\beta|) \! \! \! \sum_{|\alpha|=k+1-|\beta|} \ddfrac{1}{\alpha!} 
  \int_0^1 (1-t)^{k-|\beta|} \partial^{\alpha+\beta} u\big(\bp+t(\bx-\bp)\big)(\bx-\bp)^\alpha dt.
  \end{split}
  \]
  The steps leading to \eqref{Taylor3} result in
 \[
 \|\partial^\beta e \|_{D}^2  \le c h^{2(k+1-|\beta|)} \int_0^1 t^{-n} dt \int_{D_t} \sum_{|\alpha|=k+1-|\beta|}
 |\partial^{\alpha+\beta} u(\bx)|^2 d\bx .
 \]
 The set of indices
   $\{ \alpha+\beta,  |\alpha |=k+1-|\beta|\}$ is a subset of the set $\{\alpha,  |\alpha|=k+1\}$ implying the inequality
 $\sum_{|\alpha|=k+1-|\beta|} |\partial^{\alpha+\beta} u(\bx)|^2 \le R(\bx)$. Hence 
 \[
 \|\partial^\beta e \|_{D}^2  \le c (k+1)h^{2(k+1-|\beta|)} \int_0^1 t^{-n} dt \int_{D_t} R(\bx) d\bx .
 \] 
 We now use \eqref{Taylor8} to complete the proof. \hfill $\blacksquare$
\end{thm}
\vskip 6pt

\begin{rem}
\begin{enumerate}
\item[(i)] It is worth noting that it is the same Taylor polynomial $T_\bp^k[u]$ that is the objet of
estimate \eqref{Taylor0}
\item[(ii)] The estimate \eqref{Taylor0} with $l=k=0$ is basically a Friedrichs type inequality
with the constant Taylor polynomial in lieu of the average $\ddfrac{1}{\mu(D)} \int_D u d\bx$.
\end{enumerate}
\end{rem}

We know that there exists a well-defined trace operator $tr:H^1(D) \to L^2(\partial D)$ 
for Lipschitz domains $D$.
Of interest to us is a dimensionally ``correct''  bound on the norm $|u|_{L^2(\partial D)}$ in terms
of $\|u\|_D$ and $|u|_{1,D}$. While it does not adapt well to
{\it thin} sets, it is effective in its intended application to typical cells.
A proof that includes bounds for $|\cdot|_{L^p(\partial D}$ can be found in \cite{KJ98}. 

\begin{thm}{(Trace Inequality)}
Suppose $D$ is starlike and that assumptions (A1), (A2) are satisfied. Then
\begin{equation}\label{Trace0}
|u|_{L^2(\partial D)}^2 \le C_{tr}^{-1} \left( 2n h_D^{-1} \|u\|_{D}^2 + \ddfrac{h_D}{2} |u|_{1,D}^2\right).
\end{equation}
where $C_{tr}$ is the constant in (A2) and $h_D = {\rm diam}(D)$. 

{\bf Proof}
It suffices to consider $u\in C^\infty(\overline D)$ given that such functions are dense in $H^1(D)$ with
well defined traces in $L^2(\partial D)$. For one
\begin{equation}\label{Trace1}
\int_{\partial D} (\bx-\bp)\cdot \bfnu u^2 ds \ge C_{tr} h_D \int_{\partial D} u^2 ds 
  = C_{tr} h_D |u|^2 _{L^2(\partial D)}.
  \end{equation}
  On the other hand, the Fundamental Theorem of Calculus says that
  \begin{equation}\label{Trace2}
  \begin{split}
  \int_{\partial D} (\bx-\bp)\cdot \bfnu \, u^2 ds & = \sum_{i=1}^n \int_{\partial D} (x_i-p_i) \nu_i u^2 ds =
    \sum_{i=1}^n \int_D \ddfrac{\partial}{\partial x_i} \big( (x_i-p_i)u^2\big) d\bx \\
    & = \sum_{i=1}^n \int_D u^2 d\bx + 2 \sum_{i=1}^n \int_D (x_i-p_i) u \ddfrac{\partial u}{\partial x_i} d \bx.
    \end{split}
    \end{equation}
    The first term on the right side is $n\|u\|_{D}^2$. For the second term we have $|x_i-p_i| \le h_D$
    and use Cauchy-Schwarz to get
    \[
    \begin{split}
    \sum_{i=1}^n \int_D |x_i-p_i| u \ddfrac{\partial u}{\partial x_i} dx 
        & \le h_D \|u\|_{D} \sum_{i=1}^n \left\| \ddfrac{\partial u}{\partial x_i}\right\|_{L^2(D)} \\
        & \le \sqrt{n} \, h_D \|u\|_{D} |u|_{1,D}.
        \end{split}
        \]
        
        Combining this with \eqref{Trace2} shows that
        \begin{equation}\label{Trace3}
        \int_{\partial D} (\bx-\bp)\cdot \bfnu \, u^2 ds \le 2n \|u\|_{D}^2 + h_D ^2 |u|_{1,D}^2.
        \end{equation}
        The trace inequality \eqref{Trace0} is now a consequence of \eqref{Trace1} and \eqref{Trace3}.
        \hfill $\blacksquare$
\end{thm}
\vskip 6pt

The next result is about the approximation of a collection of real numbers by their average.
The main fact that emerges is the observation that the error can be bounded by consecutive differences
of the numbers.

\begin{lem}\label{lem1} Given $N$ real numbers $\{ \alpha_1,\dots,\alpha_N \}$
let $\beta = \frac{1}{N} \sum_{j=1}^N \alpha_j$. Then,
\begin{equation}
\sum_{j=1}^N | \alpha_j - \beta |^2 \le
   C \sum_{j=1}^{N-1} | \alpha_{j+1} - \alpha_j |^2, \label{lem1.1}
\end{equation}
where $C$ depends only on $N$.

{\bf Proof}
For any $j\in\{1,\dots,N\}$, the Cauchy--Schwarz inequality gives
\[
| \alpha_j - \beta |^2 = \frac{1}{N^2}
 \bigg| \sum_{i=1}^N (\alpha_j -\alpha_i) \bigg|^2
 \le \frac{N-1}{N^2} \sum_{i=1,j\ne i}^N | \alpha_j - \alpha_i |^2.
\]
With $j$ still fixed and $i\ne j$, we can always write $\alpha_j-\alpha_i$ as a sum of the form 
$\sum_{\ell \in L}(\alpha_{\ell+1}-\alpha_\ell)$ for some index set $L$. It follows from the above and
using the a.g.m.i. that

\begin{equation}\label{lem1.2}
 | \alpha_j - \beta |^2 \le
   C \sum_{\ell=1}^{N-1} | \alpha_{\ell+1} - \alpha_\ell |^2, \quad j=1,\dots,N,
\end{equation}
where $C$ depends only on $N$.
Inequality \eqref{lem1.1} is now a simple consequence of \eqref{lem1.2}.
\hfill $\blacksquare$
\end{lem}
\vskip 3pt

In both theory and computation of finite elements, use is made of cell-basis functions.
Estimating the sizes (norms) of these functions is of use and the issue is the subject of
our next result.
\vskip 3pt

 Let $(K, \Sigma_K, \cP_K)$ be a finite element. By that we mean a shape regular closed cell $K \in \cT_h$, 
 a finite dimensional space of functions $\cP_K$  together with a uni-solvent set
$\Sigma_K $ of linear functionals on $\cP_K$,
which are commonly referred to as degrees of freedom, (d.o.f.'s). 
We shall consider d.o.f.s having two attributes: Being attached to {\it nodes} e.g. vertices,
midpoints of edges, edges or $K$ itself and {\it types} e.g. values of functions, directional derivatives, integrals etc.

Let $\{\ba_{1,k},\dots, \ba_{\mu,k}\}$ be the set of distinct nodes of $K$.  
The set of d.o.f.'s associated with node $\ba_{i,k}$ will be denoted by 
$\{ \ell_{i,k}^j, \, j=1,\dots,\mu_i\}$. Here the superscript $j$ refers to a particular type of
d.o.f. that is asociated with $\ba_{i,k}$. Restricting ourselves to directional derivatives at nodes, we define

\[
\ell_{i,k}^j(v)= D^{s(j)} v(\ba_{i,k}) \big( \xi_1^j,\dots,\xi^j_{s(j)}\big),
\quad i=1,\dots,\mu, \ j=1,\dots, \mu_i.
\]
Here $s(j)$ is the order of derivative present and  $\xi$'s are unit vectors in $R^n$.

The canonical basis functions on $K$, $\{\phi_{i,k}^j, \ i=1,\dots,\mu, \ j=1,\dots,\mu_i\}$ for $\cP_K$
have support $K$ and are uniquely defined by
\[
\ell_{i,k}^j(\phi_{i',k}^{j'}) = \left\{ \begin{array}{lll} 1 \quad & i=i' \ {\rm and } \ j=j' \\
 0 & \quad {\rm otherwise}. \end{array} \right.
 \]

The {\it peripheral} cells $K$, i.e. those for which $\partial K \cap \partial \Omega$ has positive $(n-1)$-dimensional
measure, are allowed to have one curved face.
We assume that such cells are convex and contain an affine cell which is used to
define the nodes. For instance, they could be affine images of a reference cell $\hat K$.
On the other hand the d.o.f.'s are defined locally and are not exported from any
reference element. See e.g. Figs.  2 or 4 for further clarification.

\begin{prop}\label{prop1} Let $(K,\Sigma_K,\cP_K)$ be a finite element as described above. 
The following bounds hold for the canonical basis functions
\begin{equation}\label{prop1.1}
\| \phi_{i,k}^j \|_K \le c h_K ^ {n/2+s(j)}.
\end{equation}
where $c$ is a constant independent of the diameter $h_K$ of $K$.
\end{prop}

{\bf Proof} For $v\in \cP_K$ we define the quantity
$
\|v\|_\star = \sum_{i=1}^\mu\sum_{j=1}^{\mu_i} h_K^{s(j)} \big|\ell_{i,k}^j(v) \big|.
$
Essentially, this is a dimensionally consistent weighted sum of all the d.o.f.'s.
It is a norm in view of the uni-solvence of $\Sigma_K$ and is therefore equivalent to $ \|v\|_{L^\infty(K)}$ on $\cP_K$.
The constants in the equivalence relation may depend on dim($\cP_K$); however, a scaling argument 
\footnote[2]{we use a simple magnification and not an affine map}
shows that they are independent of $h_K$. By definition, we also have
\begin{equation}\label{prop1.2}
 \|\phi_{i,k}^j\|_\star = h_K^{s(j)}.
 \end{equation}
 
 Finally, \eqref{prop1.1} follows from \eqref{prop1.2}, the inequality 
 $\|v\|_K \le {\rm vol}(K)^{1/2} \|v\|_{L^\infty(K)}$
 and the equivalence of the aforementioned norms. \hfill $\blacksquare$
 \vskip 3pt
 
 \begin{rem}\label{rem2}
 Bounds on different norms of the basis functions are easily derived. Also, the results should remain true for d.o.f.'s
 other than directional derivatives including weighted volume- or edge/surface- integrals
 provided appropriate powers of $h_K$ are associated with them. 
 \end{rem}
 \vskip 3pt

Let $\overline \Omega = \cap_{K \in \cT_h} K$ be a partition of $\Omega$ which is {\it locally quasi-uniform},
that is, the sizes of adjacent cells are commensurate.
Let also $\cN$ denote the collection of all nodes.
For $\bfnu \in \cN$, $\omega_{\bfnu}$ will be the set of all pairs $(i,k)$ such that 
$\ba_{i,k}$ is node $i$ of cell $K$ and is equal to $\bfnu$.
In particular, the collection of $k$'s in $\omnu$  represent the cells that contain $\bfnu$.
 Its cardinality $|\omega_{\bfnu}|$ equals 1 if $\bfnu$ is in the interior of a cell or it is in the
 interior of an edge that is adjacent to the boundary of $\Omega$.
 
The global finite element space $V_h$ is the span of the functions $\phi_{\bfnu}^j:= \sum_{\ba_{i,k}=\bfnu} \phi_{i,k}^j$ 
having as support the patch of cells containing $\bfnu$, in other words, the union of the cells in $\omnu$.
Note that $V_h$ is a subspace of $\cP_K(\cT_h)$. Moreover, the local degrees of freedom are designed
from the onset to endow $\phi_{\bfnu}^j$, and consequently $V_h$, with a certain global degree of smoothness, 
e.g. $C^m, m\ge 0$ or some other type of regularity such as continuity at the midpoints of edges as is the case 
with the Crouzeix-Raviart element.
\vskip 3pt

We shall next establish the result which is the cornerstone of our approach to approximation:
A discontinuous piecewise polynomial function $v_h$ can be approximated by an element 
$\chi$ in the global space $V_h$ such that various norms of the difference $v_h-\chi$ can be bounded
purely by the jumps of directional derivatives of $v_h$ across edges of the mesh $\cT_h$.
We note that the same result, but specific to Lagrangian elements, was included in \cite{KP03}.
   
 \begin{prop}\label{prop2} Suppose we have the inclusion $ \cP_m(K) \subseteq \cP_K $ for some nonnegative $m$.
 Given $v_h \in \cP_m(\cT_h)$,  there exists $\chi \in V_h$ such that for $K \in \cT_h$,
 \begin{equation}\label{prop2.1}
 |v_h - \chi|_{l,K}^2\le c \sum_{j \in \cJ} \sum_{e \in \cE_K} h_e^{1+2s(j)-2l} |[D^{s(j)}v_h]|_e^2,\quad 0 \le l \le m,
 \end{equation}
 where $\cE_K$ denotes all the edges in $\cT_h$ that contain a vertex of $K$ \hskip -3pt
\footnote[3]{except the boundary edges, shown dashed in Figures 1, 2, 3.}
 and $\cJ$ is a set indicating the orders of directional derivatives.
 \end{prop}
 {\bf Proof} 
 We can certainly write 
 $v_h - \chi\big|_K = \sum _{i=1}^\mu \sum_{j=1}^{\mu_i}  \big( \ell_{i,k}^j(v_h) - \ell_{i,k}^j(\chi)\big)\phi_{i,K}^j$.
 which implies via  \eqref{prop1.1}
 \begin{equation}\label{prop2.2}
 |v_h-\chi|_{l,K}^2 \le c h_K^{n+2s(j)-2l } \sum _{i=1}^\mu \sum_{j=1}^{\mu_i} \big| \ell_{i,k}^j(v_h)-\ell_{i,k}^j(\chi)\big|^2.
 \end{equation}
 
 We assume that the list $\omnu$ is such that if $(i,k)$ and $(i',k')$ are consecutive pairs,
 then cells $k$ and $k'$ share an edge.
 Some observations are in order. The list may be periodic in the sense that the first and last cells
 are adjacent. This would be the case e.g. for a vertex. Also, this consecutivity is possible in two
 space dimensions but not so in higher dimensions as a counterexample shows. In that case some repetitions
 must be incorporated the effect of which will only be to increase the constants in the estimates.
 With that in mind, the value of $\chi$ is defined by averaging as follows: At each node $\bfnu=\ba_{i,k} \in \cN$
 we set
 \begin{equation}\label{prop2.3}
 \ell_{i,k}^j(\chi) = \ddfrac{1}{|\omnu|} \sum_{\ba_{i',k'} =\bfnu} \ell_{i',k'}^j(v_h), \quad j=1,\dots,\mu_{\bfnu}.
 \end{equation}
 In effect, this uniquely defines $\chi \in V_h$.
  
 Using \eqref{lem1.2} we get
 \begin{equation}\label{prop2.4}
 \big|\ell_{i,k}^j(v_h) -\ell_{i,k}^j (\chi)\big|^2 \le 
 c\sum_{l=1}^{|\omnu|-1} \big|\ell_{i,k_l}^j (v_h) -\ell_{i,k_{l+1}}^j(v_h)\big|^2 .
 \end{equation}
 Now $\ell_{i,k_l}^j (v_h) -\ell_{i,k_{l+1}}^j(v_h)$ is the jump, denoted  $[\ell_{i,k}^j(v_h)]_e$, of  $\ell_{i,k}^j$
 across the edge $e=k_l \cap k_{l+1}$. Note that the jumps are zero whenever $|\omnu|=1$. For the others,
 recalling that the d.o.f.'s are directional derivatives of type $j$, we have
  the inequality 
 \begin{eqnarray}\label{prop2.5}
 \big| [ \ell_{i,k}^j (v_h) ]_e \big| & = & \big| [ D^{s(j)} v_h(\ba_{i,k}^j) ]_e \big| \le 
     \big| [ D^{s(j)} v_h] \big|_{L^\infty(e)} \nonumber \\
 & \le &   c h_e^\frac{1-n}{2} \big| [ D^{s(j)} (v_h) ]\big|_{L^2(e)}
 \end{eqnarray} 
 holding by virtue of an $L^\infty$-$L^2$ inverse inequality on $e$. Combining \eqref{prop2.2}-\eqref{prop2.5} we obtain 
 \eqref{prop2.1}.
  \hfill $\blacksquare$

\section{Examples}
In this section we will present some cases of approximation results developed along lines previously outlined.
In the first example, as well as the others, we consider a two-dimensional partition $\cT_h$ of a domain $\Omega$ consisting
of triangles or more appropriately, tri-sided cells. We allow the peripheral triangles to have one curved side. 
We assume that the triangles are shape-regular (satisfy
the minimum angle condition)
and that the mesh is locally quasi-uniform in the sense that the diameters of adjacent
triangles are commensurate. Also, as mentioned earlier, peripheral cells contain a regular triangle which
will be used to define the nodes i.e. points.

\begin{figure}[ht]

\begin{center}
\begin{tikzpicture}[scale=2]

\draw[thick,variable=\t,domain=0:360,samples=500] plot ({2*sin(\t) }, {1.5*cos(\t)});

\foreach \i in {0,1,2,3,4,5,6,7,8,9,10,11} {
\fill ({ 2*cos(30*\i)}, {1.5*sin(30*\i)}) circle[radius=1pt];
}

\foreach \i in {0,1,2,3,4,5,6,7,8,9,10,11} { \coordinate (a\i) at ({2*cos(30*\i)}, {1.5*sin(30*\i)});}

\coordinate (a12) at (1.25,0); \fill (1.25,0) circle[radius=1pt];
\coordinate (a13) at (.75,.7);  \fill (.75,.7) circle[radius=1pt];
\coordinate (a14) at (.75,-.7); \fill (.75,-.7) circle[radius=1pt];
\coordinate (a15) at (-.5,.7);   \fill (-.5,.7) circle[radius=1pt];
\coordinate (a16) at (0,0);      \fill (0,0) circle[radius=1pt];
\coordinate (a17) at (-1,0);     \fill (-1,0) circle[radius=1pt];
\coordinate (a18) at (-.5,-.7);  \fill (-.5,-.7) circle[radius=1pt];

\draw [dashed, thick] (a0)--(a1)--(a2)--(a3)--(a4)--(a5)--(a6)--(a7)--(a8)--(a9)--(a10)--(a11)--(a0);

\draw (a0)--(a12)--(a1); 
\draw (a1)--(a13)--(a2);
\draw (a2)--(a13)--(a3);
\draw (a3)--(a15)--(a4);
\draw (a4)--(a15)--(a5);
\draw (a5)--(a17)--(a6);
\draw (a6)--(a17)--(a7);
\draw (a7)--(a18)--(a8);
\draw (a8)--(a18)--(a9);
\draw (a9)--(a14)--(a10);
\draw (a10)--(a14)--(a11);
\draw (a11)--(a12)--(a0);

\draw (a1)--(a12)--(a13);
\draw (a12)--(a13)--(a16)--(a12);
\draw (a13)--(a15)--(a16);
\draw (a15)--(a16)--(a17)--(a15);
\draw (a16)--(a17)--(a18);
\draw (a14)--(a16)--(a18)--(a14);
\draw (a12)--(a14)--(a16);

\end{tikzpicture}
\end{center}
\caption{}\label{figex1}
\end{figure}

{\bf Example 1}
We consider first the example of an element referred to as Courant's triangle. We first
establish an approximation result which is equivalent in some sense to that given by the 
Cl\'{e}ment interpolant.
It gives an $O(h)$ estimate for the $L^2$-norm of the error in terms of the $H^1$-seminorm of $u$.
Approximation results of this type have found wide application in a posteriori
estimation. The next result is a more usual one giving $O(h^2)$ error bound for functions $u \in H^2(\Omega)$.

\begin{thm}\label{thmex1} Let $\cT_h$ be the triangulation of $\Omega$ s shown in Figure 1
and let $(K, \cP_K, \Sigma_K)$ be the affine Lagrange element using as d.o.f.'s values at the vertices
leading to the space $V_h= C^0(\Omega) \cap \cP_1(\cT_h)$.

(i) Let $u\in C^0(\Omega) \cap H^1(\cT_h)$. There exists $\chi \in V_h$ such that
\begin{equation}\label{thmex1.a}
\|u-\chi\|_{K} \le c h |u|_{1,\omega_K}, \quad h = \max_{K' \in \omega_K} h_{K'}
\end{equation}
where $\omega_K$ is the collection of triangles that share a vertex with $K$.
\vskip 3pt
(ii) Let $u\in C^0(\Omega) \cap H^2(\cT_h)$. There exists $\chi \in V_h$ such that
\begin{equation}\label{thmex1.b}
|u-\chi|_{l,K} \le c h^{2-l}  |u|_{2,\omega_K}, \quad l=0,1.
\end{equation}
\end{thm}

{\bf Proof} 
(i) Let $v_h \in \cP_0(\cT_h)$ be the collection of cell-Taylor polynomials
of degree 0 on $\cT_h$ and let $\chi\in V_h$ be constructed from $v_h$ by averaging as in Proposition \ref{prop2}

Using \eqref{Taylor0} and \eqref{prop2.1} we have
\begin{equation}\label{thmex1.1}
\|u-\chi\|_K^2 \le 2 \|u-v_h\|_K^2 + 2\|v_h-\chi\|_K^2 \le c h_K^2 |u|_{1,K}^2 +c \sum_{e\in \cE_K}h_e |[v_h]|_e^2
\end{equation}
Now $[v_h]\big|_e=[v_h-u]\big|_e$  since $u$ is continuous.
Hence, using the trace inequality \eqref{Trace0} and then \eqref{Taylor0}, it follows that
\begin{eqnarray}\label{tmex1.2}
|[v_h]|_e^2 & \le & c \sum_{i=1}^2 \left( h_{K_i}^{-1} \|u-v_h\|_{K_i}^2 + h_{K_i} |u-v_h|_{1,K_i}^2\right), \quad e = K_1 \cap K_2 \\
& \le & c h |u|_{1,K}^2.   \nonumber
\end{eqnarray}
This establishes \eqref{thmex1.a}. 

The proof of \eqref{thmex1.b} follows the same steps except $v_h$ is now the piecewise affine
Taylor polynomial approximation of $u$. 
Using \eqref{Taylor0} and \eqref{prop2.1} we have for $l=0,1$
\begin{equation}\label{thmex1.2}
|u-\chi |_{l,K}^2 \le 2 |u-v_h|_{l,K} + 2|v_h-\chi |_{l,K}^2 \le c h_K^{4-2l} |u|_{2,K}^2 +c \sum_{e\in \cE_K}h_e^{1-2l} |[v_h]|_e^2.
\end{equation}
The rest of the proof uses \eqref{Trace0} and then \eqref{Taylor0} as before.
 \hfill $\blacksquare$
 \vskip 6pt
 
 We draw attention to an advantage of using Taylor approximations exemplified by the Cl\'{e}ment or
 Scott-Zhang interpolant.
 By choosing to work with a low-order Taylor approximation, we obtain upper bounds of the error
 which are lower than $k$, the order of the finite element $P_K$  yet requiring lower regularity from $u$. 
 This is a feature which is precisely not allowed by
 the Bramble-Hilbert type approach.

\vskip 6pt

 {\bf Example 2} The Argyris Triangle $(K, \cP_5(K), \Sigma_K)$ 
is a globally $C^1$, 21-degrees of freedom element. The degrees of freedom (d.o.f.'s) can be specified in a
variety of ways of which we choose
\[
\Sigma_K =\{ \partial^\alpha p(\ba_{i,K}), \ i=1,2,3, \ |\alpha| \le 2, \quad
 \nabla p(\bb_K^{ij})\cdot \bfn^{ij}, \, 1\le i<j\le 3\}
\]
where $\ba_{i,K}$ denote the vertices of $K$ and $\bb_K^{ij}$ are the midpoints of edges and $\bfn^{ij}$
is a unit normal vector to side $[\ba_{i,K}\, \ba_{j,K}]$. 


\begin{figure}[ht]     
\begin{center}
\begin{tikzpicture}[scale=2]
\draw[thick,variable=\t,domain=0:360,samples=500] plot ({2*sin(\t) }, {1.5*cos(\t)});

\foreach \i in {0,1,2,3,4,5,6,7,8,9,10,11} { 
\def\x{2*cos(30*\i};
\def\y{1.5*sin(30*\i};
\coordinate (a\i) at (\x,\y);
}
\def\x12{1.25};  \def\y12{.0};  \coordinate (a12) at (\x12,\y12);
\def\x13{.75};   \def\y13{.7};   \coordinate (a13) at (\x13,\y13);
\def\x14{.75};   \def\y14{-.7};  \coordinate (a14) at (\x14,\y14);
\def\x15{-.5};    \xdef\y15{.7}; \coordinate (a15) at (\x15,\y15); \coordinate (m) at (\x15,\y15);
\def\x16{0};      \def\y16{.0};   \coordinate (a16) at (\x16,\y16);
\def\x17{-1};     \def\y17{.0};   \coordinate (a17) at (\x17,\y17);
\def\x18{-.5};    \def\y18{-.7};  \coordinate (a18) at (\x18,\y18);

\foreach \i in {0,1,2,3,4,5,6,7,8,9,10,11,12,13,14,15,16,17,18} {
    \fill (a\i) circle[radius=1pt]; \draw[thick] (a\i) circle[radius=2pt]; \draw[thick] (a\i) circle[radius=3.5pt];
}

\draw [dashed, thick] (a0)--(a1)--(a2)--(a3)--(a4)--(a5)--(a6)--(a7)--(a8)--(a9)--(a10)--(a11)--(a0);

\draw (a0)--(a12)--(a1); 
\draw (a1)--(a13)--(a2);
\draw (a2)--(a13)--(a3);
\draw (a3)--(a15)--(a4);
\draw (a4)--(a15)--(a5);
\draw (a5)--(a17)--(a6);
\draw (a6)--(a17)--(a7);
\draw (a7)--(a18)--(a8);
\draw (a8)--(a18)--(a9);
\draw (a9)--(a14)--(a10);
\draw (a10)--(a14)--(a11);
\draw (a11)--(a12)--(a0);

\draw (a1)--(a12)--(a13);
\draw (a12)--(a13)--(a16)--(a12);
\draw (a13)--(a15)--(a16);
\draw (a15)--(a16)--(a17)--(a15);
\draw (a16)--(a17)--(a18);
\draw (a14)--(a16)--(a18)--(a14);
\draw (a12)--(a14)--(a16);


\foreach \i in {0,1,2,3,4,5,6,7,8,9,10} {
\edef\ip1{\i}
  \pgfmathparse{\ip1+1}
  \edef\ip1{\pgfmathresult}
\coordinate (m) at ($.48*(a\i)+.48*(a\ip1)$);
\coordinate (n) at  ($.52*(a\i)+.52*(a\ip1)$);
\draw[thick] (m)--(n);
}

\coordinate (m) at ($.48*(a11)+.48*(a0)$);
\coordinate (n) at  ($.52*(a11)+.52*(a0)$);
\draw[thick] (m)--(n);

\draw[thick] (1.625,-.07)--(1.625,.07);
\draw[thick] (.625,-.07)--(.625,.07);
\draw[thick] (-.5,-.07)--(-.5,.07);
\draw[thick] (-1.5,-.07)--(-1.5,.07);

\def\d{.065};
\def\s{.714}; \def\a{1}; \def\b{.35}; 
\coordinate (m)  at (\a-\d, {\s*(-\d)+\b}); 
\coordinate (n)   at (\a+\d,{\s*(\d)+\b}); \draw[thick] (m)--(n);

\def\s{-.643}; \def\a{1.491}; \def\b{.375}; 
\coordinate (m)  at (\a-\d, {\s*(-\d)+\b}); 
\coordinate (n)  at  (\a+\d,{\s*(\d)+\b}); \draw[thick] (m)--(n);

\def\s{-.417}; \def\a{.875}; \def\b{1}; 
\coordinate (m)  at (\a-\d, {\s*(-\d)+\b}); 
\coordinate (n)  at  (\a+\d,{\s*(\d)+\b}); \draw[thick] (m)--(n);

\def\s{.9375}; \def\a{.375}; \def\b{1.1}; 
\coordinate (m)  at (\a-\d, {\s*(-\d)+\b}); 
\coordinate (n)  at  (\a+\d,{\s*(\d)+\b}); \draw[thick] (m)--(n);

\def\d{.004};
\def\s{-19.64}; \def\a{1.241}; \def\b{.725}; 
\coordinate (m)  at (\a-\d, {\s*(-\d)+\b}); 
\coordinate (n)   at (\a+\d,{\s*(\d)+\b}); \draw[thick] (m)--(n);

\def\s{-20}; \def\a{.125}; \def\b{.7}; 
\coordinate (m)  at (\a-\d, {\s*(-\d)+\b}); 
\coordinate (n)  at  (\a+\d,{\s*(\d)+\b}); \draw[thick] (m)--(n);

\def\s{-20}; \def\a{-1.116}; \def\b{.725}; 
\coordinate (m)  at (\a-\d, {\s*(-\d)+\b}); 
\coordinate (n)  at  (\a+\d,{\s*(\d)+\b}); \draw[thick] (m)--(n);

\def\s{-19.64}; \def\a{1.241}; \def\b{-.725}; 
\coordinate (m)  at (\a-\d, {\s*(-\d)+\b}); 
\coordinate (n)   at (\a+\d,{\s*(\d)+\b}); \draw[thick] (m)--(n);

\def\s{-20}; \def\a{.125}; \def\b{-.7}; 
\coordinate (m)  at (\a-\d, {\s*(-\d)+\b}); 
\coordinate (n)  at  (\a+\d,{\s*(\d)+\b}); \draw[thick] (m)--(n);

\def\s{-20}; \def\a{-1.116}; \def\b{-.725}; 
\coordinate (m)  at (\a-\d, {\s*(-\d)+\b}); 
\coordinate (n)  at  (\a+\d,{\s*(\d)+\b}); \draw[thick] (m)--(n);

\def\d{.06};
\xdef\s{-.625}; \def\a{-.25}; \def\b{1.1}; 
\coordinate (m)  at (\a-\d, {\s*(-\d)+\b}); 
\coordinate (n)  at  (\a+\d,{\s*(\d)+\b}); \draw[thick] (m)--(n);

\xdef\s{.833}; \def\a{-.75}; \def\b{1}; 
\coordinate (m)  at (\a-\d, {\s*(-\d)+\b}); 
\coordinate (n)  at  (\a+\d,{\s*(\d)+\b}); \draw[thick] (m)--(n);

\xdef\s{.714}; \def\a{-.25}; \def\b{.35}; 
\coordinate (m)  at (\a-\d, {\s*(-\d)+\b}); 
\coordinate (n)  at  (\a+\d,{\s*(\d)+\b}); \draw[thick] (m)--(n);

\xdef\s{-1.071}; \def\a{.375}; \def\b{.35}; 
\coordinate (m)  at (\a-\d, {\s*(-\d)+\b}); 
\coordinate (n)  at  (\a+\d,{\s*(\d)+\b}); \draw[thick] (m)--(n);

\xdef\s{-.714}; \def\a{-.75}; \def\b{.35}; 
\coordinate (m)  at (\a-\d, {\s*(-\d)+\b}); 
\coordinate (n)  at  (\a+\d,{\s*(\d)+\b}); \draw[thick] (m)--(n);

\xdef\s{.976}; \def\a{-1.366}; \def\b{.375}; 
\coordinate (m)  at (\a-\d, {\s*(-\d)+\b}); 
\coordinate (n)  at  (\a+\d,{\s*(\d)+\b}); \draw[thick] (m)--(n);

\xdef\s{.643}; \def\a{1.491}; \def\b{-.375}; 
\coordinate (m)  at (\a-\d, {\s*(-\d)+\b}); 
\coordinate (n)  at  (\a+\d,{\s*(\d)+\b}); \draw[thick] (m)--(n);

\xdef\s{-.714}; \def\a{1}; \def\b{-.35}; 
\coordinate (m)  at (\a-\d, {\s*(-\d)+\b}); 
\coordinate (n)  at  (\a+\d,{\s*(\d)+\b}); \draw[thick] (m)--(n);

\xdef\s{1.071}; \def\a{.375}; \def\b{-.35}; 
\coordinate (m)  at (\a-\d, {\s*(-\d)+\b}); 
\coordinate (n)  at  (\a+\d,{\s*(\d)+\b}); \draw[thick] (m)--(n);

\xdef\s{-.714}; \def\a{-.25}; \def\b{-.35}; 
\coordinate (m)  at (\a-\d, {\s*(-\d)+\b}); 
\coordinate (n)  at  (\a+\d,{\s*(\d)+\b}); \draw[thick] (m)--(n);

\xdef\s{.714}; \def\a{-.75}; \def\b{-.35}; 
\coordinate (m)  at (\a-\d, {\s*(-\d)+\b}); 
\coordinate (n)  at  (\a+\d,{\s*(\d)+\b}); \draw[thick] (m)--(n);

\xdef\s{-.976}; \def\a{-1.366}; \def\b{-.375}; 
\coordinate (m)  at (\a-\d, {\s*(-\d)+\b}); 
\coordinate (n)  at  (\a+\d,{\s*(\d)+\b}); \draw[thick] (m)--(n);

\xdef\s{.417}; \def\a{.875}; \def\b{-1.}; 
\coordinate (m)  at (\a-\d, {\s*(-\d)+\b}); 
\coordinate (n)  at  (\a+\d,{\s*(\d)+\b}); \draw[thick] (m)--(n);

\xdef\s{-.9375}; \def\a{.375}; \def\b{-1.1}; 
\coordinate (m)  at (\a-\d, {\s*(-\d)+\b}); 
\coordinate (n)  at  (\a+\d,{\s*(\d)+\b}); \draw[thick] (m)--(n);

\xdef\s{.625}; \def\a{-.25}; \def\b{-1.1}; 
\coordinate (m)  at (\a-\d, {\s*(-\d)+\b}); 
\coordinate (n)  at  (\a+\d,{\s*(\d)+\b}); \draw[thick] (m)--(n);

\xdef\s{-.833}; \def\a{-.75}; \def\b{-1.}; 
\coordinate (m)  at (\a-\d, {\s*(-\d)+\b}); 
\coordinate (n)  at  (\a+\d,{\s*(\d)+\b}); \draw[thick] (m)--(n);

\end{tikzpicture}
\end{center}
\caption{Argyris element: Global triangulation}\label{figex2}
\end{figure}
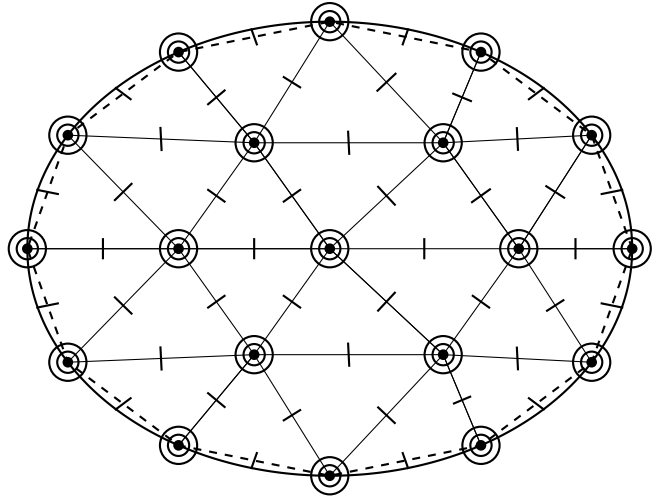

\vskip 6pt
The Argyris triangle is a finite element, c.f. Ciarlet \cite{Ciarlet}. The space $V_h$ of globally $C^1$,
piecewise quintics is constructed as usual. 

\begin{thm}\label{thmex2}
Let $u \in C^2(\Omega) \cap H^6(\cT_h)$. there exists $\chi \in V_h$ such that for $0 \le l \le 6$
\begin{equation}\label{thmex2.1}
|u-\chi|_{l,K} \le c h^{6-l} |u|_{6,\omega_K}, \quad h = \max_{K' \in \omega_K} h_{K'}
\end{equation}
{\bf Proof} 
Let $v_h \in \cP_5(\cT_h)$ be the collection of cell-Taylor polynomials
of degree 5 on $\cT_h$ and let $\chi\in V_h$ be constructed from $v_h$ by averaging as in Proposition \ref{prop2}

Using \eqref{Taylor0} and \eqref{prop2.1} we have
\begin{eqnarray}\label{thmex2.2}
|u-\chi|_{l,K}^2 & \le & 2 |u-v_h|_{l,K}^2 + 2|v_h-\chi|_{l,K}^2  \\ \nonumber & \le & c h_K^{12-2l} |u|_{6,K}^2 
         + c   \sum_{j\in \cJ} \sum_{e\in \cE_K}h_e^{1+2s(j)-2l}|[D^{s(j)}v_h]|_e^2 .
\end{eqnarray}
The directional derivatives in this case have order $0\le j\le 2$. Hence the jumps $[D^{s(j)}u]_e$ 
 vanish since $u$ belongs to $C^2(\Omega)$.
 
Thus, using the trace inequality \eqref{Trace0} and then \eqref{Taylor0}, it follows that
\begin{align*}
|[D^{s(j)} v_h]|_e^2 &= |[D^{s(j)} (v_h-u)]|_e^2 \le c \sum_{i=1}^2 \left( h_{K_i}^{-1-2s(j) } \|u-v_h\|_{K_i}^2  \right. \\
      & \left. + h_{K_i}^{1+2s(j)}  |u-v_h|_{1,K_i}^2\right) \le c h^{11-2s(j)}  \sum_{i=1}^2|u|_{6,K_i}^2 
\end{align*}
$e=K_1\cap K_2$. Using this in \eqref{thmex2.2} establishes \eqref{thmex2.1}. \hfill $\blacksquare$

\end{thm}
\vskip 6pt

\begin{rem} As in the case of  \eqref{thmex1.a}, we can use Taylor approximations of degree lower than 5.
The space $V_h$ being the same as above, the powers of $h$ are lower than 6 but so is the dependence
on the order of the seminorms of $u$.
\end{rem}
\vskip 6pt

{\bf Example 3} The Hsieh-Clough-Tocher (HCT) triangle.

Is also known as a {\it composite} or {\it macro element} in the sense that the set $P_K$ is composed
of piecewise cubic polynomials.
The triangle $K$ is divided into 3 triangles $K_i, \, i=1,2,3$ the point $\ba_K$ being interior
to $K$. On each triangle $K_i$ we have a cubic polynomial which are then assembled (``stitched'')
in a $C^1$ fashion and the global finite element space $V_h$ is also of class $C^1$. 
Figure (a) below should serve as a guide to the assembly of the 3 cubic polynomials into the
HCT macro element. These represent 18 equations of the 30 d.o.f.'s of the 3 cubics;
hence there are always nontrivial solutions. It is rather straightforward to show
that all these solutions are in $C^1(K)$.

\begin{table}[ht]
\begin{equation}\label{HCT}
\begin{split}
 & P_K = \{p \in C^1(K); \ p\big|_{K_i} \in \cP_3(K_i), \ 1\le i \le 3\}; \ \mbox{dim} P_K =12 \\
& \Sigma_K =\{p(\ba_i), \partial_1 p(\ba_i), \partial_2 p(\ba_i), \partial_\nu p(\bb_i), 1 \le i \le 3\}.
\end{split}
\end{equation}
\caption{Hsieh-Clough-Tocher element}\label{table1}
\end{table}

A proof that this is indeed a finite element can be found in \cite{Ciarlet} pp. 341--342.
Essentially, 18 equations are used to stitch the 3 cubics into composite functions in $C^1(K)$
with 12 d.o.f.'s  that are listed in \eqref{HCT}.
 \vskip 12pt
 
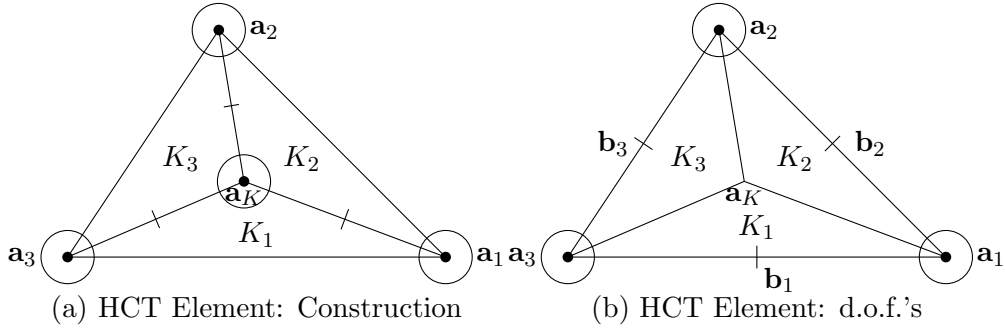
 \begin{figure}[ht]
  \begin{minipage}{.45\linewidth}
\begin{tikzpicture}
\draw[solid] (0,0)--(5,0)--(2,3)--(0,0);
\draw[solid] (0,0)--(7./3.,1)--(5,0);  \draw[solid] (7/3,1)--(2,3);
\fill (0,0) circle[radius=2pt]; \draw (0,0) circle[radius=10pt];
\fill (5,0) circle[radius=2pt]; \draw (5,0) circle[radius=10pt];
\fill (2,3) circle[radius=2pt]; \draw (2,3) circle[radius=10pt];
\fill (7/3,1) circle[radius=2pt]; \draw (7/3,1) circle[radius=10pt]; 
\node at (7/3,.8) {$\ba_K$}; 
 \node at (-.6,0) {$\ba_3$}; 
 \node at (5.6,0) {$\ba_1$};
 \node at (2.6,3) {$\ba_2$};
 
\def\alpha{.05};
\draw[solid] (7/6-\alpha,.5+7*\alpha/3)--(7/6+\alpha,.5-7*\alpha/3);
\draw[solid] (22/6+\alpha,.5+8*\alpha/3)--(22/6-\alpha,.5-8*\alpha/3);
\draw[solid] (13/6+2*\alpha,2+\alpha/3)--(13/6-2*\alpha,2-\alpha/3);
\node at (2.5,-.7) {\rm (a) HCT Element: Construction};
\node at (2.5,.3) {$K_1$}; \node at (3.1,1.3) {$K_2$}; \node at (1.5,1.3) {$K_3$};
\end{tikzpicture}
\end{minipage}
\begin{minipage}{.45\linewidth}
\begin{tikzpicture}
\draw[solid] (0,0)--(5,0)--(2,3)--(0,0);
\draw[solid] (0,0)--(7./3.,1)--(5,0);  \draw[solid] (7/3,1)--(2,3);
\fill (0,0) circle[radius=2pt]; \draw (0,0) circle[radius=10pt];
\fill (5,0) circle[radius=2pt]; \draw (5,0) circle[radius=10pt];
\fill (2,3) circle[radius=2pt]; \draw (2,3) circle[radius=10pt];
\node at (7/3,.8) {$\ba_K$};  
 \node at (-.6,0) {$\ba_3$};  \node at (2.8,-.3) {$\bb_1$};
 \node at (5.6,0) {$\ba_1$}; \node at (4,1.5) {$\bb_2$};
 \node at (2.6,3) {$\ba_2$}; \node at (.6,1.5) {$\bb_3$};
 
\def\alpha{.15};  \draw[solid] (2.5,\alpha)--(2.5,-\alpha);
\def\alpha{.1};    \draw[solid] (3.5+\alpha,1.5+\alpha)--(3.5-\alpha,1.5-\alpha);
\def\alpha{.04};  \draw[solid] (1-3*\alpha,1.5+2*\alpha)--(1+3*\alpha,1.5-2*\alpha);
\node at (2.5,-.7) {\rm (b) HCT Element: d.o.f.'s};
\node at (2.5,.4) {$K_1$}; \node at (3.0,1.3) {$K_2$}; \node at (1.6,1.3) {$K_3$};
\end{tikzpicture}
\end{minipage}
 \end{figure}
 
\vskip 8pt
A very interesting feature of the HCT and some other composite elements is that they 
give rise to globally
$C^1$ subspaces, yet their d.o.f.'s do not involve second-order partial derivatives. In fact, this would have
been impossible to achieve with polynomials as a deep result of \v{Z}eni\v{s}ek shows 
 \cite{Zenisek73, Zenisek74}. This feature of composite elements has been shown to be very
 useful in a a variety of contexts including domain decomposition \cite{GHV2009, KC2018}. 
 See also \cite{DDPS79} for the construction of such composite elements.
 
 As shown in figure \ref{figex3.2}, we impose a triangulation $\cT_h$ on $\Omega$.
 We construct a global space $V_h $ of HCT elements which is a subspace of $C^1(\Omega) \cap \cP_K(\cT_h)$.

\begin{figure}[ht]
\begin{center}
\begin{tikzpicture}[scale=2]

\draw[thick,variable=\t,domain=0:360,samples=500] plot ({2*sin(\t) }, {1.5*cos(\t)});

\foreach \i in {0,1,2,3,4,5,6,7,8,9,10,11} { 
\def\x{2*cos(30*\i};
\def\y{1.5*sin(30*\i};
\coordinate (a\i) at (\x,\y);
}
\def\x12{1.25};  \def\y12{.0};  \coordinate (a12) at (\x12,\y12);
\def\x13{.75};   \def\y13{.7};   \coordinate (a13) at (\x13,\y13);
\def\x14{.75};   \def\y14{-.7};  \coordinate (a14) at (\x14,\y14);
\def\x15{-.5};    \xdef\y15{.7}; \coordinate (a15) at (\x15,\y15); \coordinate (m) at (\x15,\y15);
\def\x16{0};      \def\y16{.0};   \coordinate (a16) at (\x16,\y16);
\def\x17{-1};     \def\y17{.0};   \coordinate (a17) at (\x17,\y17);
\def\x18{-.5};    \def\y18{-.7};  \coordinate (a18) at (\x18,\y18);

\foreach \i in {0,1,2,3,4,5,6,7,8,9,10,11,12,13,14,15,16,17,18} {
    \fill (a\i) circle[radius=1pt]; \draw[thick] (a\i) circle[radius=2pt];
}

\draw [dashed, thick] (a0)--(a1)--(a2)--(a3)--(a4)--(a5)--(a6)--(a7)--(a8)--(a9)--(a10)--(a11)--(a0);

\draw (a0)--(a12)--(a1); 
\draw (a1)--(a13)--(a2);
\draw (a2)--(a13)--(a3);
\draw (a3)--(a15)--(a4);
\draw (a4)--(a15)--(a5);
\draw (a5)--(a17)--(a6);
\draw (a6)--(a17)--(a7);
\draw (a7)--(a18)--(a8);
\draw (a8)--(a18)--(a9);
\draw (a9)--(a14)--(a10);
\draw (a10)--(a14)--(a11);
\draw (a11)--(a12)--(a0);

\draw (a1)--(a12)--(a13);
\draw (a12)--(a13)--(a16)--(a12);
\draw (a13)--(a15)--(a16);
\draw (a15)--(a16)--(a17)--(a15);
\draw (a16)--(a17)--(a18);
\draw (a14)--(a16)--(a18)--(a14);
\draw (a12)--(a14)--(a16);


\foreach \i in {0,1,2,3,4,5,6,7,8,9,10} {
\edef\ip1{\i}
  \pgfmathparse{\ip1+1}
  \edef\ip1{\pgfmathresult}
\coordinate (m) at ($.48*(a\i)+.48*(a\ip1)$);
\coordinate (n) at  ($.52*(a\i)+.52*(a\ip1)$);
\draw[thick] (m)--(n);
}

\coordinate (m) at ($.48*(a11)+.48*(a0)$);
\coordinate (n) at  ($.52*(a11)+.52*(a0)$);
\draw[thick] (m)--(n);

\draw[thick] (1.625,-.07)--(1.625,.07);
\draw[thick] (.625,-.07)--(.625,.07);
\draw[thick] (-.5,-.07)--(-.5,.07);
\draw[thick] (-1.5,-.07)--(-1.5,.07);

\def\d{.065};
\def\s{.714}; \def\a{1}; \def\b{.35}; 
\coordinate (m)  at (\a-\d, {\s*(-\d)+\b}); 
\coordinate (n)   at (\a+\d,{\s*(\d)+\b}); \draw[thick] (m)--(n);

\def\s{-.643}; \def\a{1.491}; \def\b{.375}; 
\coordinate (m)  at (\a-\d, {\s*(-\d)+\b}); 
\coordinate (n)  at  (\a+\d,{\s*(\d)+\b}); \draw[thick] (m)--(n);

\def\s{-.417}; \def\a{.875}; \def\b{1}; 
\coordinate (m)  at (\a-\d, {\s*(-\d)+\b}); 
\coordinate (n)  at  (\a+\d,{\s*(\d)+\b}); \draw[thick] (m)--(n);

\def\s{.9375}; \def\a{.375}; \def\b{1.1}; 
\coordinate (m)  at (\a-\d, {\s*(-\d)+\b}); 
\coordinate (n)  at  (\a+\d,{\s*(\d)+\b}); \draw[thick] (m)--(n);

\def\d{.004};
\def\s{-19.64}; \def\a{1.241}; \def\b{.725}; 
\coordinate (m)  at (\a-\d, {\s*(-\d)+\b}); 
\coordinate (n)   at (\a+\d,{\s*(\d)+\b}); \draw[thick] (m)--(n);

\def\s{-20}; \def\a{.125}; \def\b{.7}; 
\coordinate (m)  at (\a-\d, {\s*(-\d)+\b}); 
\coordinate (n)  at  (\a+\d,{\s*(\d)+\b}); \draw[thick] (m)--(n);

\def\s{-20}; \def\a{-1.116}; \def\b{.725}; 
\coordinate (m)  at (\a-\d, {\s*(-\d)+\b}); 
\coordinate (n)  at  (\a+\d,{\s*(\d)+\b}); \draw[thick] (m)--(n);

\def\s{-19.64}; \def\a{1.241}; \def\b{-.725}; 
\coordinate (m)  at (\a-\d, {\s*(-\d)+\b}); 
\coordinate (n)   at (\a+\d,{\s*(\d)+\b}); \draw[thick] (m)--(n);

\def\s{-20}; \def\a{.125}; \def\b{-.7}; 
\coordinate (m)  at (\a-\d, {\s*(-\d)+\b}); 
\coordinate (n)  at  (\a+\d,{\s*(\d)+\b}); \draw[thick] (m)--(n);

\def\s{-20}; \def\a{-1.116}; \def\b{-.725}; 
\coordinate (m)  at (\a-\d, {\s*(-\d)+\b}); 
\coordinate (n)  at  (\a+\d,{\s*(\d)+\b}); \draw[thick] (m)--(n);

\def\d{.06};
\xdef\s{-.625}; \def\a{-.25}; \def\b{1.1}; 
\coordinate (m)  at (\a-\d, {\s*(-\d)+\b}); 
\coordinate (n)  at  (\a+\d,{\s*(\d)+\b}); \draw[thick] (m)--(n);

\xdef\s{.833}; \def\a{-.75}; \def\b{1}; 
\coordinate (m)  at (\a-\d, {\s*(-\d)+\b}); 
\coordinate (n)  at  (\a+\d,{\s*(\d)+\b}); \draw[thick] (m)--(n);

\xdef\s{.714}; \def\a{-.25}; \def\b{.35}; 
\coordinate (m)  at (\a-\d, {\s*(-\d)+\b}); 
\coordinate (n)  at  (\a+\d,{\s*(\d)+\b}); \draw[thick] (m)--(n);

\xdef\s{-1.071}; \def\a{.375}; \def\b{.35}; 
\coordinate (m)  at (\a-\d, {\s*(-\d)+\b}); 
\coordinate (n)  at  (\a+\d,{\s*(\d)+\b}); \draw[thick] (m)--(n);

\xdef\s{-.714}; \def\a{-.75}; \def\b{.35}; 
\coordinate (m)  at (\a-\d, {\s*(-\d)+\b}); 
\coordinate (n)  at  (\a+\d,{\s*(\d)+\b}); \draw[thick] (m)--(n);

\xdef\s{.976}; \def\a{-1.366}; \def\b{.375}; 
\coordinate (m)  at (\a-\d, {\s*(-\d)+\b}); 
\coordinate (n)  at  (\a+\d,{\s*(\d)+\b}); \draw[thick] (m)--(n);

\xdef\s{.643}; \def\a{1.491}; \def\b{-.375}; 
\coordinate (m)  at (\a-\d, {\s*(-\d)+\b}); 
\coordinate (n)  at  (\a+\d,{\s*(\d)+\b}); \draw[thick] (m)--(n);

\xdef\s{-.714}; \def\a{1}; \def\b{-.35}; 
\coordinate (m)  at (\a-\d, {\s*(-\d)+\b}); 
\coordinate (n)  at  (\a+\d,{\s*(\d)+\b}); \draw[thick] (m)--(n);

\xdef\s{1.071}; \def\a{.375}; \def\b{-.35}; 
\coordinate (m)  at (\a-\d, {\s*(-\d)+\b}); 
\coordinate (n)  at  (\a+\d,{\s*(\d)+\b}); \draw[thick] (m)--(n);

\xdef\s{-.714}; \def\a{-.25}; \def\b{-.35}; 
\coordinate (m)  at (\a-\d, {\s*(-\d)+\b}); 
\coordinate (n)  at  (\a+\d,{\s*(\d)+\b}); \draw[thick] (m)--(n);

\xdef\s{.714}; \def\a{-.75}; \def\b{-.35}; 
\coordinate (m)  at (\a-\d, {\s*(-\d)+\b}); 
\coordinate (n)  at  (\a+\d,{\s*(\d)+\b}); \draw[thick] (m)--(n);

\xdef\s{-.976}; \def\a{-1.366}; \def\b{-.375}; 
\coordinate (m)  at (\a-\d, {\s*(-\d)+\b}); 
\coordinate (n)  at  (\a+\d,{\s*(\d)+\b}); \draw[thick] (m)--(n);

\xdef\s{.417}; \def\a{.875}; \def\b{-1.}; 
\coordinate (m)  at (\a-\d, {\s*(-\d)+\b}); 
\coordinate (n)  at  (\a+\d,{\s*(\d)+\b}); \draw[thick] (m)--(n);

\xdef\s{-.9375}; \def\a{.375}; \def\b{-1.1}; 
\coordinate (m)  at (\a-\d, {\s*(-\d)+\b}); 
\coordinate (n)  at  (\a+\d,{\s*(\d)+\b}); \draw[thick] (m)--(n);

\xdef\s{.625}; \def\a{-.25}; \def\b{-1.1}; 
\coordinate (m)  at (\a-\d, {\s*(-\d)+\b}); 
\coordinate (n)  at  (\a+\d,{\s*(\d)+\b}); \draw[thick] (m)--(n);

\xdef\s{-.833}; \def\a{-.75}; \def\b{-1.}; 
\coordinate (m)  at (\a-\d, {\s*(-\d)+\b}); 
\coordinate (n)  at  (\a+\d,{\s*(\d)+\b}); \draw[thick] (m)--(n);

\end{tikzpicture}
\end{center}
\caption{HCT. Global triangulation}\label{figex3.2}
\end{figure}
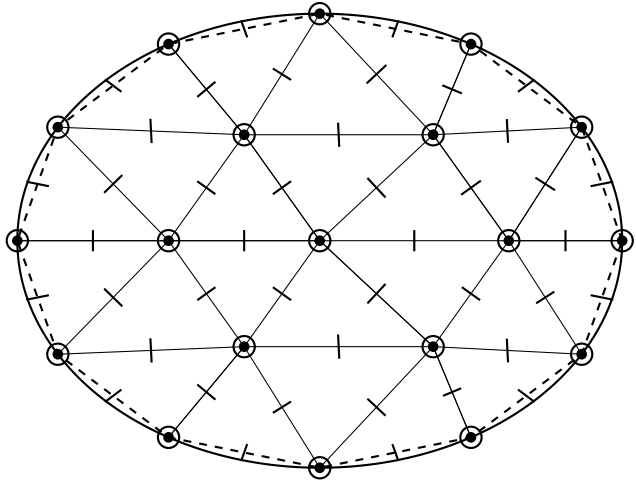

We now present our approximation result concerning the HCT element. The proof follows basically the same 
lines already seen in the previous two examples.
\vskip 6pt

\begin{thm}\label{thmex3}
Let $u \in C^1(\Omega) \cap H^4(\cT_h)$. there exists $\chi \in V_h$ such that for l=0,\dots,4
\begin{equation}\label{thmex3.1}
|u-\chi|_{l,K} \le c h^{4-l} |u|_{4,\omega_K}, \quad h = \max_{K' \in \omega_K} h_{K'}
\end{equation}
{\bf Proof} 
Let $v_h \in \cP_3(\cT_h)$ be the collection of cell-Taylor polynomials
of degree 3 of $u$ on $\cT_h$. Note that $\cP_K$ contains $\cP_3(K)$ and hence
$\chi \in V_h$ can be constructed by averaging as in Proposition \ref{prop2}

Using \eqref{Taylor0} and \eqref{prop2.1}, for l=0,\dots,4  we have
\begin{eqnarray}\label{thmex3.2}
|u-\chi|_{l,K}^2 & \le & 2 |u-v_h|_{l,K}^2 + 2|v_h-\chi|_{l,K}^2  \\ \nonumber 
& \le & c h_K^{8-2l} |u|_{4,K}^2   + c   \sum_{j\in \cJ} \sum_{e\in \cE_K}h_e^{1+2s(j)-2l}|[D^{s(j)}v_h]|_e^2 .
\end{eqnarray}
The directional derivatives in this case have order $0\le j\le 1$. Hence the jumps $[D^{s(j)}u]_e$ 
 vanish since $u$ belongs to $C^1(\Omega)$.
 Thus, using the trace inequality \eqref{Trace0} and then \eqref{Taylor0}, it follows that 
\begin{align*}
|[D^{s(j)} v_h]|_e^2 &= |[D^{s(j)}( v_h-u)]|_e^2  \le  c \sum_{i=1}^2 \left( h_{K_i}^{-1-2s(j) } \|u-v_h\|_{K_i}^2 \right. \\
        & \left. + h_{K_i}^{1+2s(j)} |u-v_h|_{1,K_i}^2\right) \le c h^{7-2s(j)} \sum_{i=1}^2 |u|_{4,K_i}^2
\end{align*}
$e=K_1 \cap K_2$.
Using this in \eqref{thmex3.2} establishes \eqref{thmex3.1}. \hfill $\blacksquare$

\end{thm}
\vskip 6pt

\begin{rem}
The proof (pp.~344--346) in Ciarlet \cite{Ciarlet} is for the
interpolant and is local, i.e., \eqref{thmex3.1} holds in $L^p$-based norms
on $u$ on $K$ rather than $\omega_K$.
It involves embedding the HCT element into an affine equivalent
family with attention paid to the location of the point $\ba_K$. 
The treatment is quite intricate taking 3 pages to complete.

Another point worth mentioning would arise in applying the techniques of \cite{EG2017} to the HCT element.
Their general assumption is the inclusion $\hat \cP \subset W^{k+1,\infty}(\hat K^, R^q)$ with $k=3$ in the case at
hand. However this inclusion does not hold for the HCT element.

\end{rem}
\vskip 6pt

The subject of prescribing boundary conditions on the space $V_h$ is
not considered here in any detail. Suffice it to say that there are some
ideas that appear reasonable and could be explored more fully in a future work.
For one, if $\Omega$ is polygonal, homogeneous boundary condition can be
incorporated relatively easily. More general conditions can be handled by Nitsche's \cite{Nitsche71}
method of using boundary penalty terms. This is an idea that apparently goes back to J.L. Lions. 
Recall that this approach was a precursor of the Discontinuous Galerkin method.
We promptly point out that we are not doing DG here and in some sense the whole approach
is antithetical to it.



\begin{thebibliography}{99}
\bibitem{adams75}
{\sc R.~A. Adams},
{\it Sobolev Spaces},
Academic Press, New York, 1975.

\bibitem{BJK90}
{\sc G.~A. Baker, W.~N. Jureidini, and O.~A. Karakashian},
{\it Piecewise solenoidal vector fields and the Stokes problem},
SIAM J. Numer. Anal., 27 (1990), pp.~1466--1485.

\bibitem{brennerxx}{\sc S. Brenner} Private communication.

\bibitem{BS2008}{\sc S. C. Brenner and L. R. Scott}
{\it The Mathematical theory of Finite Element Methods, third ed.}
Texts in Applied Mathematics, Springer (2008).

\bibitem{Ciarlet} {\sc P. G. Ciarlet} {\it The Finite Element Method for Elliptic Problems},
North Holland, (1978)

\bibitem{CW71}{\sc P. G. Ciarlet, C. Wagschal}
{\it Multipoint Taylor formulas and applications to the finite element method}, Numer. Math., 17,
pp.~84-100.


\bibitem{Clement75}{\sc P. Cl\'{e}ment}
{\it Approximation by finite element functions using local regularization}, RAIRO Anal. Num\'{e}r. 9, (1975)
pp.~77--84.

\bibitem{DL1953}{\sc J. Deny, J. L. Lions}
{\it Les espaces de type Beppo Levi}
Ann. Institut Fourier, (Grenoble) (1953-1954), pp.~305--370.

\bibitem{DDPS79}{\sc J. Douglas Jr., T. Dupont, T. Percell, L.R. Scott}, {\it A family of $C^1$ finite elements
with optimal approximation properties for various Galerkin methods for 2nd and 4th order problems},
RAIRO-Anal. Num\' {e}r. {\bf 13} (3) (1979), pp.~227--255.

\bibitem{DS80}{\sc T. Dupont, L.R. Scott}
{\it Polynomial approximation of functions in Sobolev spaces}, Math. Comp., 34, (1980), pp.~441--463.

\bibitem{Duran}{\sc R. G. Duran}, {\it On polynomial Approximation in Sobolev Spaces}, SIAM J. Numer. Anal.
Vol. 20, No. 5, October 1983, pp. 985--988.

\bibitem{EG2017}{\sc A. Ern, J.-L. Guermond}, {\it Finite element quasi-interpolation and best approximation},
ESAIM: Mathematical Modeling and Numerical Analysis, (2017).

\bibitem{GHV2009}{\sc E. Georgoulis, P. Houston}, {\it Discontinuous Galerkin methods for the biharmonic
problem}, IMA J. Numer. Anal., {\bf 29} (3), (2009), pp.~573--594.

\bibitem{KJ98}{\sc O. Karakashian, W. N. Jureidini}, {\it A nonconforming Finite Element Method for
the Stationary Navier-Stokes Equations}, SIAM J. Numer. Anal., 35, No. 1, (1998), pp.~93--120.

\bibitem{KP03}{\sc O. Karakashian, F. Pascal}, {\it A posteriori error estimates for a discontinuous Galerkin
approximation of second-order elliptic problems}, SIAM J. Numer. Anal., 41, No. 6, (2003), pp.~2374--2399.

\bibitem{KC2018}{\sc O. Karakashian, C. Collins}, {\it Two-level additive Schwarz methods for discontinuous
Galerkin approximations of the biharmonic equation}, Jour. of Sci. Comput., {\bf 74}, (1),  (2018), pp.~573--604.

\bibitem{Nitsche71}{\sc J.A. Nitsche}, {\it \"{U}ber ein Variationprinzip zur L\"{o}sung von Dirichlet-problemen
bei Verwendung von Teilr\"{a}umen, die keinen Randbedingungen unterworfen sind,}
Abh. Math. Sem. Univ. Hamburg {\bf 36}, 9-15.

\bibitem{oswald93}{\sc P. Oswald}
{\it On a BPX-preconditioner for P1 elements}, Computing, 51(2), (1993), pp.~125--133.

\bibitem{SZ90}{\sc L.R. Scott, S. Zhang}
{\it Finite element interpolation of non-smooth functions satisfying boundary conditions}, Math. Comp., 54,
(1990), pp.~483--493.

\bibitem{Stein}{\sc E. Stein}
{\it Singular integrals and differentiability properties of functions}, Princeton University Press, Princeton, 
New Jersey, (1970).

\bibitem{Zenisek73}{\sc A. \v{Z}eni\v{s}ek}, {\it Polynomial approximation on tetrahedrons in the finite
element method}, J. Approximation Theory 7, (1973) pp.~334--351.

\bibitem{Zenisek74}{\sc A. \v{Z}eni\v{s}ek}, {\it A general theorem on triangular finite $C^{(m)}$ elements},
Rev. Fran\c{c}aise Automaat. Informat. Recherche Op\'{e}rationelle S\'{e}r. Rouge Anaal. Num\'{e}r.,
{\bf R}-2, (1974), pp.~119--127.







\end{thebibliography}
\end{document}